\documentclass[12pt, twoside]{article}
\usepackage{secdot}

\usepackage{amsmath, amsthm, amscd, amsfonts, amssymb, graphicx, color}
\usepackage[bookmarksnumbered, colorlinks, plainpages]{hyperref}

\begin{document}
\centerline{}

\centerline{\Large{\bf Almost Kenmotsu metric as Ricci-Yamabe soliton}}

\centerline{}

\newcommand{\mvec}[1]{\mbox{\bfseries\itshape #1}}

\centerline{\large{ Dibakar Dey}}




\newtheorem{theorem}{\quad Theorem}[section]

\newtheorem{definition}[theorem]{\quad Definition}

\newtheorem{proposition}[theorem]{\quad Proposition}

\newtheorem{question}[theorem]{\quad Question}

\newtheorem{remark}[theorem]{\quad Remark}

\newtheorem{corollary}[theorem]{\quad Corollary}

\newtheorem{note}[theorem]{\quad Note}

\newtheorem{lemma}[theorem]{\quad Lemma}
\newtheorem{example}[theorem]{\quad Example}
\newtheorem{notation}[theorem]{\quad Notation}
\numberwithin{equation}{section}
\newcommand{\be}{\begin{equation}}
\newcommand{\ee}{\end{equation}}
\newcommand{\bea}{\begin{eqnarray}}
\newcommand{\eea}{\end{eqnarray}}

\vspace{0.5cm}
\begin{abstract}
 The object of the present paper is to characterize two classes of almost Kenmotsu manifolds admitting Ricci-Yamabe soliton. It is shown that a $(k,\mu)'$-almost Kenmotsu manifold admitting a Ricci-Yamabe soliton or gradient Ricci-Yamabe soliton is locally isometric to the Riemannian product $\mathbb{H}^{n+1}(-4) \times \mathbb{R}^n$. For the later case, the potential vector field is pointwise collinear with the Reeb vector field. 
Also, a $(k,\mu)$-almost Kenmotsu manifold admitting certain Ricci-Yamabe soliton with the curvature property $Q \cdot P = 0$ is locally isometric to the hyperbolic space $\mathbb{H}^{2n+1}(-1)$ and the non-existense of the curvature property $Q \cdot R = 0$ is proved.
\end{abstract} 


\section{Introduction }
In 1982, the concept of Ricci flow was introduced by Hamilton \cite{hamilton1}. The Ricci flow is an evolution equation for metrics on a Riemannian manifold $(M^n,g)$ given by
\bea
\nonumber \frac{\partial g}{\partial t} = - 2S,
\eea
where $g$ is the Riemannian metric and $S$ denotes the $(0,2)$-symmetric Ricci tensor. A self similar solution of the Ricci flow is called Ricci soliton \cite{hamilton2} and is defined as
\bea
\nonumber \pounds_V g + 2S = 2\lambda g
\eea
where $\pounds_V$ denotes the Lie derivative along the vector field $V$ and $\lambda$ is a constant.\\

\noindent The notion of Yamabe flow was proposed by Hamilton \cite{hamilton3} in 1989, which is defined on a Riemannian manifold $(M^n,g)$ as 
\bea
\nonumber  \frac{\partial g}{\partial t} = - rg,
\eea
where $r$ is the scalar curvature of the manifold. In dimension 2, we know that $S = \frac{r}{2}g$ and hence, the Ricci flow and Yamabe flow are equivalent. For $n > 2$, these two notions are not same as the Yamabe flow preserves the conformal class of metrics but the Ricci flow does not in general. A Yamabe soliton on a Riemannian manifold $(M^n,g)$ is a triplet $(g,V,\lambda)$ satisfying
\bea
\nonumber \frac{1}{2}\pounds_V g = (\lambda - r)g,
\eea
where $\lambda$ is a constant. The Ricci soliton and Yamabe soliton are said to be shrinking, steady or expanding according as $\lambda$ is positive, zero or negative respectively.\\

\noindent In 2019, G$\ddot{u}$ler and Crasmareanu \cite{gc} introduced a scalar combination of the Ricci flow and the Yamabe flow under the name of Ricci-Yamabe map. Let $(M^n,g)$ be a Riemannian manifold, $T_2^s(M)$ be the linear space of its $(0,2)$-symmetric tensor fields and $Riem(M) \subset T_2^s(M)$ be the infinite space of its Riemannian metrics. In \cite{gc}, the authors gives the following two definitions.
\begin{definition} $(\cite{gc})$
A Riemannian flow on $M$ is a smooth map:
$$g : I \subseteq \mathbb{R} \rightarrow Riem(M),$$
where $I$ is a given open interval.
\end{definition}
\begin{definition} $(\cite{gc})$
The map $RY^{(\alpha,\beta,g)} : I \rightarrow T_2^s(M)$ given by:
$$RY^{(\alpha,\beta,g)} = \frac{\partial g}{\partial t}(t) + 2\alpha S(t) + \beta r(t)g(t)$$
is called the $(\alpha,\beta)$-Ricci-Yamabe map of the Riemannian flow $(M,g)$. If $$RY^{(\alpha,\beta,g)} \equiv 0,$$ then $g(\cdot)$ will be called an $(\alpha,\beta)$-Ricci-Yamabe flow.
\end{definition}

\noindent The $(\alpha,\beta)$-Ricci-Yamabe map can be a Riemannian or semi-Riemannian or singular Riemannian flow due to the signs of $\alpha$ and $\beta$. The freedom of choosing the sign of $\alpha$ and $\beta$ is very useful in geometry and relativity. Recently, a bi-metric approach of the spacetime geometry appears in \cite{akrami} and \cite{bc}. The notion of $(\alpha,\beta)$-Ricci-Yamabe soliton or simply 
Ricci-Yamabe soliton from the Ricci-Yamabe flow can be defined as follows:
\begin{definition}
A Riemannian manifold $(M^n,g)$, $n > 2$ is said to admit a Ricci-Yamabe soliton (in short, RYS) $(g,V,\lambda,\alpha,\beta)$ if
\bea
\pounds_V g + 2\alpha S = (2\lambda - \beta r)g, \label{1.1}
\eea
where $\lambda,\; \alpha,\; \beta \in \mathbb{R}$.
If $V$ is gradient of some smooth function $f$ on $M$, then the above notion is called gradient Ricci-Yamabe solition (in short, GRYS) and then (\ref{1.1}) reduces to
\bea
\nabla^2 f + \alpha S = (\lambda - \frac{1}{2}\beta r)g, \label{1.2}
\eea
where $\nabla^2 f$ is the Hessian of $f$.
\end{definition}

\noindent The RYS (or GRYS) is said to be expanding, steady or shrinking according as $\lambda < 0$, $\lambda = 0$ or $\lambda > 0$ respectively. A RYS (or GRYS) is called an almost RYS (or GRYS) if $\alpha$, $\beta$ and $\lambda$ are smooth functions on $M$. The above notion generalizes a large class of soliton like equations. A RYS (or GRYS) is said to be a 
\begin{itemize}
    \item Ricci soliton (or gradient Ricci soliton) (see \cite{hamilton2}) if $\alpha = 1$, $\beta = 0$.
    \item Yamabe soliton (or gradient Yamabe soliton) (\cite{hamilton3} if $\alpha = 0$, $\beta = 1$.
    \item Einstein soliton (or gradient Einstein soliton) (\cite{catino1}) if $\alpha = 1$, $\beta = -1$.
    \item $\rho$-Einstein soliton (or gradient $\rho$-Einstein soliton) (\cite{catino2}) if $\alpha = 1$, $\beta = - 2\rho$.
\end{itemize}
We say that the RYS or GRYS is proper if $\alpha \neq 0,\;1$.\\

\noindent In 2016, Wang \cite{ywang} proved that a $(k,\mu)'$-almost Kenmotsu manifold admitting gradient Ricci soliton is locally isometric to $\mathbb{H}^{n+1}(-4) \times \mathbb{R}^n$. In \cite{vv}, the authors proved the same result for a gradient $\rho$-Einstein soliton. Thus a natural question is
\begin{question}
Does the above result is true for a $(2n + 1)$-dimensional $(k,\mu)'$-almost Kenmotsu manifold admitting RYS or GRYS?
\end{question}
We will answer this question affirmatively. Moreover, we have studied this notion on $(k,\mu)$-almost Kenmotsu manifolds with some curvature properties. The paper is organized as follows: In section 2, we give some preliminaries on almost Kenmotsu manifolds. Section 3 deals with $(k,mu)'$-almost Kenmotsu manifolds admitting RYS and GRYS. Section 4 contains some curvature properties of $(k,\mu)$-almost Kenmotsu manifolds admitting proper RYS.

\section{Preliminaries}
A $(2n+1)$-dimensional almost contact metric manifold $M$ is a smooth manifold together with a structure $(\phi, \xi, \eta,g)$ satisfying
\be
\phi^{2}X = -X + \eta(X)\xi,\; \eta(\xi)=1,\; \phi \xi = 0,\; \eta \circ \phi = 0 \label{2.1}
\ee
\be
 g(\phi X,\phi Y)=g(X,Y)-\eta(X)\eta(Y) \label{2.2}
\ee
for any vector fields $X$, $Y$ on $M$, where $\phi$ is a $(1,1)$-tensor field, $\xi$ is a unit vector field, $\eta$ is a 1-form defined by $\eta(X) = g(X,\xi)$ and $g$ is the Riemannian metric. Using (\ref{2.2}), we can easily see that $\phi$ is skew symmetric, that is,
\be
g(\phi X,Y) = - g(X,\phi Y). \label{2.3}
\ee
The fundamental 2-form $\Phi$ on an almost contact metric manifold is defined by $\Phi(X,Y)=g(X,\phi Y)$ for any vector fields $X$, $Y$ on $M$. An almost contact metric manifold such that $\eta$ is closed and $d\Phi = 2\eta \wedge \Phi$ is called almost Kenmotsu manifold (see \cite{dp}, \cite{pv}). Let us denote the distribution orthogonal to $\xi$ by $\mathcal{D}$ and defined by $\mathcal{D} = Ker(\eta) = Im(\phi)$. In an almost Kenmotsu manifold, since $\eta$ is closed, $\mathcal{D}$ is an integrable distribution. \\
 Let $M$ be a $(2n + 1)$-dimensional almost Kenmotsu manifold. We denote by $h = \frac{1}{2}\pounds_{\xi}\phi$ and $l = R(\cdot, \xi)\xi$ on $M$. The tensor fields $l$ and $h$ are symmetric operators and satisfy the following relations \cite{pv}:
 \bea
h\xi=0,\;l\xi=0,\;tr(h)=0,\;tr(h\phi)=0,\;h\phi+\phi h=0, \label{2.4}
 \eea
\bea
 \nabla_{X}\xi=X - \eta(X)\xi - \phi hX(\Rightarrow \nabla_{\xi}\xi=0), \label{2.5}
\eea
 \be
  \phi l \phi-l = 2(h^{2} - \phi^{2}),\label{2.6}
 \ee
 \be
   R(X,Y)\xi = \eta(X)(Y - \phi hY) - \eta(Y)(X - \phi hX)+(\nabla_{Y}\phi h)X - (\nabla_{X}\phi h)Y \label{2.7}
 \ee
 for any vector fields $X,Y$ on $M$. The $(1,1)$-type symmetric tensor field $h'=h\circ\phi$ is anti-commuting with $\phi$ and $h'\xi=0$. Also it is clear that (\cite{dp}, \cite{wang-liu1})
 \bea
  h=0\Leftrightarrow h'=0,\;\;h'^{2}=(k+1)\phi^2(\Leftrightarrow h^{2}=(k+1)\phi^2).\label{2.8}
\eea
The notion of $(k,\mu)$-nullity distribution on a contact metric manifold $M$ was introduced by Blair et. al. \cite{bkp}, which is defined for any $p \in M$ and $k, \mu \in \mathbb{R}$ as follows:
\bea
\nonumber N_p(k,\mu) &=& \{ Z \in T_p(M) : R(X,Y)Z = k[g(Y,Z)X - g(X,Z)Y] \\ &&+ \mu[g(Y,Z)hX - g(X,Z)hY]\} \label{2.9}
\eea
for any vector fields $X,\; Y \in T_p(M)$, where $T_p(M)$ denotes the tangent space of $M$ at any point $p \in M$ and $R$ is the Riemann curvature tensor.
In \cite{dp}, Dileo and Pastore introduced the notion of $(k,\mu)'$-nullity distribution, on an almost  Kenmotsu manifold  $(M, \phi, \xi, \eta, g)$, which is defined for any $p \in M$ and $k,\mu \in \mathbb {R}$ as follows:
\bea
 N_{p}(k,\mu)' &=& \{Z\in T_{p}(M):R(X,Y)Z = k[g(Y,Z)X-g(X,Z)Y]\nonumber\\&&+\mu[g(Y,Z)h'X-g(X,Z)h'Y]\}.\label{2.10}
 \eea
 for any vector fields $X,\; Y \in T_p(M)$. For further details on almost Kenmotsu manifolds, we refer the reader to go through the references (\cite{dey2}-\cite{dp}, \cite{ywang}-\cite{wang-liu2}).

\section{$(k,\mu)'$-almost Kenmotsu manifolds}
   Let $X\in\mathcal{D}$ be the eigen vector of $h'$ corresponding to the eigen value $\delta$. Then from (\ref{2.8}), it is clear that $\delta^{2} = -(k+1)$, a constant. Therefore $k \leq -1$ and $\delta = \pm\sqrt{-k-1}$. We denote by $[\delta]'$ and $[-\delta]'$ the corresponding eigen spaces related to the non-zero eigen value $\delta$ and $-\delta$ of $h'$, respectively. In \cite{dp}, it is proved that in a $(2n + 1)$-dimensional $(k,\mu)'$-almost Kenmotsu manifold $M$ with $h' \neq 0$, $k < -1,\;\mu = -2$ and Spec$(h') = \{0,\delta,-\delta\}$, with $0$ as simple eigen value and $\delta = \sqrt{-k-1}$. 
 In \cite{wang-liu2}, Wang and Liu proved that for a $(2n + 1)$-dimensional $(k,\mu)'$-almost Kenmotsu manifold $M$ with $h' \neq 0$, the Ricci operator $Q$ of $M$ is given by
\bea
QX = -2nX + 2n(k+1)\eta(X)\xi - 2nh'X.\label{3.1}
\eea
Moreover, the scalar curvature of $M$ is $r = 2n(k-2n)$. From (\ref{2.10}), we have
\bea
R(X,Y)\xi=k[\eta(Y)X-\eta(X)Y] - 2[\eta(Y)h'X-\eta(X)h'Y],\label{3.2}
\eea
where
$k,\mu\in\mathbb R$. Also from (\ref{3.2}), we get
\bea
R(\xi,X)Y = k[g(X,Y)\xi - \eta(Y)X] - 2[g(h'X,Y)\xi - \eta(Y)h'X].\label{3.3}
 \eea
Again from (\ref{3.1}), we have
\bea
S(Y,\xi) = g(QY,\xi) = 2nk\eta(Y). \label{3.4}
\eea
Using (\ref{2.5}), we obtain
\bea
(\nabla_X \eta)Y = g(X,Y) - \eta(X)\eta(Y) + g(h'X,Y). \label{3.5}
\eea
\subsection{Ricci-Yamabe soliton}
We now consider the notion of RYS in the framework  $(k,\mu)'$-almost Kenmotsu manifolds. To prove our first theorem regarding RYS, we need the following lemmas:
\begin{lemma} $(\cite{dey1})$\label{l3.1}
In a $(k,\mu)'$-almost Kenmotsu manifold $M^{2n+1}$  with $h' \neq 0$, the following relation holds
\bea
\nonumber &&(\nabla_Z S)(X,Y) - (\nabla_X S)(Y,Z) - (\nabla_Y S)(X,Z) \\ \nonumber &&= - 4n(k + 2)g(h'X,Y)\eta(Z). 
\eea
\end{lemma}

\begin{lemma} $(\cite{dey1})$\label{l3.2}
In a $(k,\mu)'$-almost Kenmotsu manifold $M^{2n+1}$, $(\pounds_X h')Y = 0$ for any $X,\;Y \in [\delta]'$ or $X,\;Y \in [- \delta]'$, where Spec$(h') = \{0,\delta,-\delta\}$.
\end{lemma}

\begin{theorem} \label{t3.3}
 If $(g,V,\lambda,\alpha,\beta)$ be a RYS on a $(2n + 1)$-dimensional $(k,\mu)'$-almost Kenmotsu manifold $M$, then the manifold $M$ is locally isometric to $\mathbb{H}^{n+1}(-4)$ $\times$ $\mathbb{R}^n$, provided $2\lambda - \beta r \neq 4nk\alpha$.
\end{theorem}
\begin{proof} From (\ref{1.1}) we have
 \bea
 (\pounds_V g)(X,Y) + 2\alpha S(X,Y) = [2\lambda - \beta r]g(X,Y). \label{3.6}
 \eea
Differentiating the foregoing equation covariantly along any vector field $Z$, we obtain
\bea
(\nabla_Z \pounds_V g)(X,Y) = - 2\alpha(\nabla_Z S)(X,Y). \label{3.7}
\eea
Due to Yano \cite{yano}, we have the following commutation formula
\be
\nonumber (\pounds_V \nabla_X g - \nabla_X \pounds_V g - \nabla_{[V,X]}g)(Y,Z) = -g((\pounds_V \nabla)(X,Y),Z) - g((\pounds_V \nabla)(X,Z),Y).
\ee
Since $\nabla g = 0$, then the above relation becomes
\bea
(\nabla_X \pounds_V g)(Y,Z) = g((\pounds_V \nabla)(X,Y),Z) + g((\pounds_V \nabla)(X,Z),Y). \label{3.8}
\eea
Since $\pounds_V \nabla$ is symmetric, then it follows from (\ref{3.8}) that
\bea
\nonumber g((\pounds_V \nabla)(X,Y),Z) &=& \frac{1}{2}(\nabla_X \pounds_V g)(Y,Z) + \frac{1}{2}(\nabla_Y \pounds_V g)(X,Z) \\ && - \frac{1}{2}(\nabla_Z \pounds_V g)(X,Y). \label{3.9}
\eea
Using (\ref{3.7}) in (\ref{3.9}), we have
\be
g((\pounds_V \nabla)(X,Y),Z) = \alpha[(\nabla_Z S)(X,Y) - (\nabla_X S)(Y,Z) - (\nabla_Y S)(X,Z)]. \label{3.10}
\ee
 Using lemma \ref{l3.1} in (\ref{3.10}) yields 
\bea
\nonumber g((\pounds_V \nabla)(X,Y),Z) = - 4n\alpha(k + 2)g(h'X,Y)\eta(Z),
\eea
which implies
\bea
 (\pounds_V \nabla)(X,Y) = - 4n\alpha(k + 2)g(h'X,Y)\xi. \label{3.11}
\eea
Putting $Y = \xi$ in (\ref{3.11}), we get $(\pounds_V \nabla)(X,\xi) = 0$ and this implies \\$\nabla_Y (\pounds_V \nabla)(X,\xi) = 0$. Therefore,
\bea
(\nabla_Y \pounds_V \nabla)(X,\xi) + (\pounds_V \nabla)(\nabla_Y X,\xi) + (\pounds_V \nabla)(X,\nabla_Y \xi) = 0. \label{3.12}
\eea
Using $(\pounds_V \nabla)(X,\xi) = 0$, (\ref{3.11}) and (\ref{2.5}) in (\ref{3.12}), we obtain 
\bea
(\nabla_Y \pounds_V \nabla)(X,\xi) = 4n\alpha(k + 2)[g(h'X,Y) + g(h'^2X,Y)]\xi. \label{3.13}
\eea
Now, it is well known that (see \cite{yano})
\bea
\nonumber (\pounds_V R)(X,Y)Z = (\nabla_X \pounds_V \nabla)(Y,Z) - (\nabla_Y \pounds_V \nabla)(X,Z),
\eea
We now use (\ref{3.13}) in the foregoing equation to obtain
\bea
(\pounds_V R)(X,\xi)\xi = (\nabla_X \pounds_V \nabla)(\xi,\xi) - (\nabla_\xi \pounds_V \nabla)(X,\xi) = 0. \label{3.14}
\eea
Now, setting $Y = \xi$ in (\ref{3.6}) and using (\ref{3.4}) we have
\bea
(\pounds_V g)(X,\xi)  = [2\lambda - \beta r - 4nk\alpha]\eta(X), \label{3.15}
\eea
which implies
\bea
(\pounds_V \eta)X - g(X,\pounds_V \xi) = [2\lambda - \beta r - 4nk\alpha]\eta(X). \label{3.16}
\eea
 Putting $X = \xi$ in (\ref{3.16}), we can easily obtain  \bea
\eta(\pounds_V \xi) = - \frac{1}{2}[2\lambda - \beta r - 4nk\alpha]. \label{3.17}
\eea
From (\ref{3.2}), we can write 
\bea
 R(X,\xi)\xi = k(X - \eta(X)\xi) - 2h'X. \label{3.18}
\eea
Now, using (\ref{3.16})-(\ref{3.18}) and (\ref{3.2})-(\ref{3.3}), we obtain
\bea
\nonumber (\pounds_V R)(X,\xi)\xi &=& \pounds_V R(X,\xi)\xi - R(\pounds_V X,\xi)\xi - R(X,\pounds_V \xi)\xi - R(X,\xi)\pounds_V \xi \\ \nonumber &=&  k[2\lambda - \beta r - 4nk\alpha](X - \eta(X)\xi) - 2(\pounds_V h')X \\ \nonumber &&- 2[2\lambda - \beta r - 4nk\alpha]h'X - 2\eta(X)h'(\pounds_V \xi) \\ && - 2g(h'X,\pounds_V \xi)\xi. \label{3.19}
\eea
Equating (\ref{3.14}) and (\ref{3.19}) and then taking inner product with $Y$, we get 
\bea
\nonumber && k[2\lambda - \beta r - 4nk\alpha](g(X,Y) - \eta(X)\eta(Y))  \\ \nonumber && - 2g((\pounds_V h')X,Y) - 2[2\lambda - \beta r - 4nk\alpha]g(h'X,Y)  \\   &&- 2\eta(X)g(h'(\pounds_V \xi),Y) - 2g(h'X,\pounds_V \xi)\eta(Y) = 0. \label{3.20}
\eea
Replacing $X$ by $\phi X$ in the foregoing equation, we get 
\bea
\nonumber && k[2\lambda - \beta r - 4nk\alpha]g(\phi X,Y) -  2g((\pounds_V h')\phi X,Y) \\ && - 2[2\lambda - \beta r - 4nk\alpha]g(h' \phi X,Y) = 0. \label{3.21}
\eea
Let $X \in [- \delta]'$ and $V \in [\delta]'$, then $\phi X \in [\delta]'$. Then from (\ref{3.21}), we have
\bea
(k - 2\delta)[2\lambda - \beta r - 4nk\alpha]g(\phi X,Y) - 2g((\pounds_V h')\phi X,Y) = 0. \label{3.22}
\eea
Since, $V,\;\phi X \in [\delta]'$, using lemma \ref{l3.2}, we have $(\pounds_V h')\phi X = 0$. Therefore, equation (\ref{3.22}) reduces to
\bea
\nonumber (k - 2\delta)[2\lambda - \beta r - 4nk\alpha]g(\phi X,Y) = 0,
\eea
which implies  $k = 2\delta$, since by hypothesis $2\lambda - \beta r \neq 4nk\alpha$.\\
  Now $k = 2\delta$ and  $\delta^{2} = -(k+1)$ together implies $\delta = - 1$ and hence $k = -2$. Then from Proposition 4.2 of \cite{dp}, we have
$\alpha = - 1$ and hence, $k = -2$. Then from proposition 4.2 of \cite{dp}, we have $R(X_\delta,Y_\delta)Z_\delta = 0$ and
\bea
 \nonumber 	R(X_{-\delta},Y_{-\delta})Z_{-\delta} = - 4[g(Y_{-\delta},Z_{-\delta})X_{-\delta} - g(X_{-\delta},Z_{-\delta})Y_{-\delta}],
  \eea
  for any $X_\delta,\;Y_\delta,\;Z_\delta \in [\delta]'$ and $X_{-\delta},\;Y_{-\delta},\;Z_{-\delta} \in [-\delta]'$. Also noticing $\mu = -2$ it follows from proposition 4.3 of \cite{dp} that $K(X,\xi) = -4$ for any $X \in [-\delta]' $ and $K(X,\xi) = 0$ for any $X \in [\delta]' $. Again from proposition 4.3 of \cite{dp} we see that $K(X,Y) = -4$ for any $X, Y \in [-\delta]' $ and $K(X,Y) = 0$ for any $X, Y \in [\delta]' $. As shown in \cite{dp} that the distribution $[\xi] \oplus [\delta]'$ is integrable with totally geodesic leaves and the distribution $[-\delta]'$ is integrable with totally umbilical leaves by $H = -(1 - \delta)\xi$, where $H$ is the mean curvature tensor field for the leaves of $[-\delta]'$ immersed in $M$. Here $\delta = -1$, then the two orthogonal distributions $[\xi] \oplus [\delta]'$ and $[-\delta]'$ are both integrable with totally geodesic leaves immersed in $M$. Then we can say that $M$ is locally isometric to $\mathbb{H}^{n+1}(-4)$ $\times$ $\mathbb{R}^n$.
 \end{proof}
 
 \begin{remark}
 If $2\lambda - \beta r - 4nk\alpha = 0$ and since $r = 2n(k - 2n)$, then $\lambda = n\beta(k - 2n) + 2nk\alpha$. Since $k < -1$, if $\alpha, \beta > 0$, then $\lambda < 0$ and hence, the RYS is expanding. If $\alpha, \beta < 0$, then $\lambda > 0$ and the RYS is shrinking. If $n\beta(k - 2n) + 2nk\alpha = - 2nk\alpha$, then $\lambda = 0$ and the RYS is steady.
 \end{remark}

\subsection{Gradient Ricci-Yamabe soliton}
   We now consider the notion of GRYS in the framework of $(k,\mu)'$-almost Kenmotsu manifolds and extend the preceding theorem \ref{t3.3} by considering $V$ as a gradient vector field. In this regard, the following theorem is proved.

\begin{theorem}\label{t3.5}
If $(g,V,\lambda,\alpha,\beta)$ be a GRYS on a $(2n + 1)$-dimensional $(k,\mu)'$-almost Kenmotsu manifold $M$, then the manifold $M$ is locally isometric to $\mathbb{H}^{n+1}(-4) \times \mathbb{R}^n$ or $V$ is pointwise collinear with the characteristic vector field $\xi$.
\end{theorem}

\begin{proof} Let $V$ be the gradient of a non-zero smooth function $f : M \rightarrow \mathbb{R}$, that is, $V = Df$, where $D$ is the gradient operator. Then from (\ref{1.2}), we can write
\bea
\nabla_X Df = (\lambda - \frac{1}{2}\beta r)X - \alpha QX. \label{3.23}
\eea
Differentiating this covariantly along any vector field $Y$, we obtain
\bea
\nabla_Y \nabla_X Df = (\lambda - \frac{1}{2}\beta r)\nabla_Y X - \alpha \nabla_Y QX. \label{3.24}
\eea
Interchanging $X$ and $Y$ in (\ref{3.24}) yields
\bea
\nabla_X \nabla_Y Df = (\lambda - \frac{1}{2}\beta r)\nabla_X Y - \alpha \nabla_X QY. \label{3.25}
\eea
From (\ref{3.23}), we get
\bea
\nabla_{[X,Y]} Df = (\lambda - \frac{1}{2}\beta r)(\nabla_X Y - \nabla_Y X) - \alpha Q(\nabla_X Y - \nabla_Y X). \label{3.26}
\eea
It is well known that 
\bea
\nonumber R(X,Y)Df = \nabla_X \nabla_Y Df - \nabla_Y \nabla_X Df - \nabla_{[X,Y]} Df
\eea
Substituting (\ref{3.24})-(\ref{3.26}) in the foregoing equation yields
\bea
R(X,Y)Df = \alpha[(\nabla_Y Q)X - (\nabla_X Q)Y]. \label{3.27}
\eea
 With the help of (\ref{3.1}), we obtain
\bea
\nonumber (\nabla_Y Q)X &=& \nabla_Y QX - Q(\nabla_Y X) \\ \nonumber &=& 2n(k + 1)(g(X,Y) - \eta(X)\eta(Y) + g(h'X,Y))\xi \\ \nonumber &&+ 2n(k + 1)\eta(X)(Y - \eta(Y)\xi - \phi hY) + 2ng(h'Y + h'^2Y,X)\xi \\ \nonumber && + 2n\eta(X)(h'Y + h'^2Y). 
\eea
Interchanging $X$ and $Y$ in the preceding equation we will obtain $(\nabla_X Q)Y$. Now, substituting $(\nabla_Y Q)X$ and $(\nabla_X Q)Y$ in (\ref{3.27}) and using (\ref{2.8}), we obtain 
\bea
R(X,Y)Df = 2n\alpha(k + 2)[\eta(X)h'Y - \eta(Y)h'X]. \label{3.28} 
\eea
Putting $X = \xi$ in (\ref{3.28}) and then taking inner product with $X$ yields
\bea
 g(R(\xi,Y)Df,X) = 2n\alpha(k + 2)g(h'Y,X). \label{3.29}
\eea
Again, using (\ref{3.3}), we have
\bea
\nonumber g(R(\xi,Y)Df,X) &=& - g(R(\xi,Y)X,Df)\\ \nonumber &=& -kg(X,Y)(\xi f) + k\eta(X)(Y f)\\ && + 2g(h'X,Y)(\xi f) - 2\eta(X)((h'Y)f). \label{3.30}
\eea
Equating (\ref{3.29}) and (\ref{3.30}) and then antisymmetrizing, we get
\bea
\nonumber k\eta(X)(Y f) - k\eta(Y)(X f) - 2\eta(X)((h'Y)f) + 2\eta(Y)((h'X)f) = 0. 
\eea
Replacing $X$ by $\xi$ in the above equation, we have
\bea
\nonumber k(Yf) - k(\xi f)\eta(Y) - 2(h'Y)f) = 0,
\eea
which implies
\bea
k[Df - (\xi f)\xi] - 2h'(Df) = 0. \label{3.31}
\eea
Operating $h'$ on (\ref{3.31}) and using (\ref{2.8}), we obtain
\bea
h'(Df) = - \frac{2(k + 1)}{k}[Df - (\xi f)\xi]. \label{3.32}
\eea
Substituting (\ref{3.32}) in (\ref{3.31}), we get
\bea
\nonumber (k + 2)^2[Df - (\xi f)\xi] = 0,
\eea
which implies either $k = - 2$ or $Df = (\xi f)\xi$.\\
If $k = - 2$, then by the same argument as earlier, the manifold $M$ is locally isometric to $\mathbb{H}^{n+1}(-4) \times \mathbb{R}^n$.\\
If $V = Df = (\xi f)\xi$, then $V$ is pointwise collinear with $\xi$. This completes the proof.
\end{proof}

\noindent To obtain some consequences of the above theorem, we need the following definition:
\begin{definition}
A vector field $V$ on an almost contact metric manifold $M$ is said to be an infinitesimal contact transformation if $\pounds_V \eta = \psi \eta$ for some smooth function $\psi$ on $M$. In particular, if $\psi = 0$, then $V$ is called a strict infinitesimal contact transformation.
 \end{definition}

\begin{corollary}
If $(g,V,\lambda,\alpha,\beta)$ be a GRYS on a $(2n + 1)$-dimensional $(k,\mu)'$-almost Kenmotsu manifold $M$ with $k \neq -2$, then
\begin{itemize}
    \item[$(1)$] The potential vector field $V$ is a constant multiple of $\xi$.
    \item[$(2)$] $V$ is a strict infinitesimal contact transformation.
    \item[$(3)$] $V$ leaves $h'$ invariant.
\end{itemize}
\end{corollary}
\begin{proof}
Since $k \neq - 2$, then from theorem \ref{t3.5}, we have $V = (\xi f)\xi = b\xi$, where $(\xi f) = b$ is some smooth function on $M$. Then using (\ref{2.5}), we can easily obtain
\bea
\nonumber (\pounds_{b\xi}g)(X,Y) &=& (Xb)\eta(Y) + (Yb)\eta(X) \\ && + 2b[g(X,Y) - \eta(X)\eta(Y) - g(\phi hX,Y)]. \label{3.33}
\eea
Substituting (\ref{3.33}) in (\ref{3.6}), we obtain
\bea
\nonumber && (Xb)\eta(Y) + (Yb)\eta(X) + 2b[g(X,Y) - \eta(X)\eta(Y) - g(\phi hX,Y)] \\ && = (2\lambda - \beta r)g(X,Y) - 2\alpha S(X,Y). \label{3.34}
\eea
Putting $X = Y = \xi$ in (\ref{3.34}) and using (\ref{3.4}), we obtain
\bea
2(\xi b) = 2\lambda - \beta r - 4nk\alpha. \label{3.35}
\eea
Let $\{e_i\}$ be an orthonormal basis of the tangent space at each point of $M$. Substituting $X = Y = e_i$ in (\ref{3.34}) and then summing over $i$, we obtain
\bea
2(\xi b) = (2\lambda - \beta r)(2n + 1) - 2\alpha r - 4nb. \label{3.36}
\eea
Since $\alpha$, $\beta$, $\lambda$ and $r$ is constant, then equating (\ref{3.35}) and (\ref{3.36}), we can easily see that $b$ is constant. This proves $(1)$.\\
Since $b$ is constant, then from (\ref{3.35}), we have
\bea
2\lambda - \beta r = 4nk\alpha. \label{3.37}
\eea
Again since $b$ is constant and $V = b\xi$, then $\pounds_V \xi = 0$. Using (\ref{3.37}) in (\ref{3.15}), we get ($\pounds_V g)(X,\xi) = 0$, which implies $(\pounds_V \eta)X = 0$ for any vector field $X$. This proves $(2)$.\\
Now using $\pounds_V \xi = 0$ and (\ref{3.37}) in (\ref{3.20}, we get $(\pounds_V h')X = 0$ for any vector field $X$, which means $V$ leaves $h'$ invariant. This proves $(3)$.
\end{proof}

\begin{remark}
It is known that for a smooth tensor field $T$, $\pounds_X T = 0$ if and only if $\phi_t$ is a symmetric transformation for $T$, where $\{\phi_t : t \in \mathbb{R}\}$ is the $1$-parameter grouop of diffeomorphisms corresponding to the vector field $X$ on a manifold (see \cite{mdo}). Since $h'$ is a smooth tensor field of type $(1,1)$ on $M$, then $\pounds_V h' = 0$ if and only if $\psi_t$ is a symmetric transformation for $h'$, where $\{\psi_t : t \in \mathbb{R}\}$ is the $1$-parameter group of diffeomorphisms corresponding to the vector field $V$.
\end{remark}

\section{$(k,\mu)$-almost Kenmotsu manifolds}
In this section, we study RYS on $(k,\mu)$-almost Kenmotsu manifolds with some curvature properties. From (\ref{2.9}), we have
\bea
R(X,Y)\xi = k[\eta(Y)X - \eta(X)Y] + \mu[\eta(Y)hX - \eta(X)hY]. \label{4.1}
\eea
In \cite{dp}, Dileo and Pastore proved that for a $(k,\mu)$-almost Kenmotsu manifold, $k = -1$ and $h = 0$. Hence, from (\ref{4.1}), we get the followings:
\bea
R(X,Y)\xi = \eta(X)Y - \eta(Y)X. \label{4.2}
\eea
\bea
R(\xi,X)Y = - g(X,Y)\xi + \eta(Y)X. \label{4.3}
\eea
\bea
S(X,\xi) = - 2n\eta(X) \;\; \mathrm{and}\;\; Q\xi = - 2n\xi. \label{4.4}
\eea
Also from (\ref{2.5}), we get
\bea
\nabla_X \xi = X - \eta(X)\xi. \label{4.5}
\eea
Again in \cite{dp}, it is proved that in an almost Kenmotsu manifold with $\xi$ belonging to the $(k,\mu)$-nullity distribution, the sectional curvature $K(X,\xi) = - 1$. From this, we get $r = -2n(2n + 1)$.

\begin{lemma} \label{l4.1}
If $(g,\xi,\lambda,\alpha,\beta)$ be a proper RYS on a $(2n + 1)$-dimensional $(k,\mu)$-almost Kenmotsu manifold $M$, then the manifold $M$ is $\eta$-Einstein.
\end{lemma}
\begin{proof}
Considering $V = \xi$ in (\ref{1.1}), then we get
\bea
(\pounds_\xi g)(X,Y) + 2\alpha S(X,Y) = (2\lambda - \beta r)g(X,Y). \label{4.6}
\eea
Now,
\bea
\nonumber (\pounds_\xi g)(X,Y) = g(\nabla_X \xi,Y) + g(\nabla_Y \xi,X).
\eea
Using (\ref{4.5}) in the foregoing equation yields
\bea
(\pounds_\xi g)(X,Y) = 2[g(X,Y) - \eta(X)\eta(Y)]. \label{4.7}
\eea
Substituting (\ref{4.7}) in (\ref{4.6}) and using the properness of the RYS and $r = -2n(2n + 1)$, we obtain
\bea
S(X,Y) = \frac{1}{\alpha}[\lambda +n\beta(2n + 1) - 1]g(X,Y) + \frac{1}{\alpha}\eta(X)\eta(Y). \label{4.8}
\eea
This proves that the manifold is $\eta$-Einstein.
\end{proof}

\noindent Now from (\ref{4.4}), we have $$S(\xi,\xi) = -2n.$$
Again from (\ref{4.8}), we get 
$$S(\xi,\xi) = \frac{1}{\alpha}[\lambda + n\beta(2n + 1) - 1] + \frac{1}{\alpha}.$$
Equating these two values of $S(\xi,\xi)$, we obtain
\bea
\lambda +n\beta(2n + 1) = -2n\alpha. \label{4.9}
\eea
Using (\ref{4.9}) in (\ref{4.8}), we have
\bea
S(X,Y) = - \frac{1}{\alpha}(2n\alpha + 1)g(X,Y) + \frac{1}{\alpha}\eta(X)\eta(Y), \label{4.10}
\eea
which implies
\bea
QX = - \frac{1}{\alpha}(2n\alpha + 1)X + \frac{1}{\alpha}\eta(X)\xi. \label{4.11}
\eea
\subsection{Proper RYS and the curvature condition $Q \cdot P = 0$}
We now aim to study a proper RYS on a $(2n + 1)$-dimensional $(k,\mu)$-almost Kenmotsu manifold admitting the curvature property $Q \cdot P = 0$, where $P$ is the projective curvature tensor defined for a $(2n + 1)$-dimensional Riemannian manifold as
\bea
P(X,Y)Z = R(X,Y)Z - \frac{1}{2n}[S(Y,Z)X - S(X,Z)Y]. \label{4.12}
\eea
\begin{theorem} \label{t4.2}
If $(g,\xi,\lambda,\alpha,\beta)$ is a proper RYS on a $(2n + 1)$-dimensional $(k,\mu)$-almost Kenmotsu manifold $M$ satisfying the curvature property $Q \cdot P = 0$, then the manifold $M$ is locally isometric to the hyperbolic space $\mathbb{H}^{2n+1}(-1)$.
\end{theorem}
\begin{proof}
Let us suppose that the curvature property $Q \cdot P = 0$ holds on $M$. Then for any vector fields $X,\; Y,\; Z$ on $M$, we have
\bea
\nonumber Q(P(X,Y)Z) - P(QX,Y)Z - P(X,QY)Z - P(X,Y)QZ = 0.
\eea
Using (\ref{4.11}) in the foregoing equation yields 
\bea
\nonumber && \frac{1}{\alpha}[\eta(P(X,Y)Z)\xi - \eta(X)P(\xi,Y)Z - \eta(Y)P(X,\xi)Z \\ &&- \eta(Z)P(X,Y)\xi + 2(2n\alpha + 1)P(X,Y)Z] = 0. \label{4.13}
\eea
Now, with the help of (\ref{4.2})-(\ref{4.4}), we calculate the followings:
\bea
P(\xi,Y)Z = - g(Y,Z)\xi - \frac{1}{2n}S(Y,Z)\xi. \label{4.14}
\eea
\bea
P(X,\xi)Z = g(X,Z)\xi + \frac{1}{2n}S(X,Z)\xi. \label{4.15}
\eea
\bea
P(X,Y)\xi = 0. \label{4.16}
\eea
\bea
\nonumber \eta(P(X,Y)Z) &=& \eta(Y)g(X,Z) - \eta(X)g(Y,Z) \\ && - \frac{1}{2n}[S(Y,Z)\eta(X) - S(X,Z)\eta(Y)]. \label{4.17}
\eea
Substituting (\ref{4.14})-(\ref{4.17}) in (\ref{4.13}) yields
\bea
\nonumber 2(2n\alpha + 1)P(X,Y)Z = 0,
\eea
which implies either $2n\alpha + 1 = 0$ or $P(X,Y)Z = 0$.\\
If $2n\alpha + 1 = 0$, then from (\ref{4.10}), we get 
\bea
\nonumber S(X,Y) = \frac{1}{\alpha}\eta(X)\eta(Y) = -2n\eta(X)\eta(Y),
\eea
which implies $r = -2n$, a contradiction to the fact that $r = - 2n(2n + 1)$. Hence, $P(X,Y)Z = 0$ and therefore, (\ref{4.12}) implies 
\bea
R(X,Y)Z = \frac{1}{2n}[S(Y,Z)X - S(X,Z)Y]. \label{4.18}
\eea
Taking inner product of (\ref{4.18}) with $W$ and then contracting $Y$ and $Z$ yields 
\bea
S(X,W) = - 2ng(X,W). \label{4.19}
\eea
Using (\ref{4.19}) in (\ref{4.18}), we obtain
\bea
\nonumber R(X,Y)Z = -[g(Y,Z)X - g(X,Z)Y].
\eea
This proves that $M$ is locally isometric to $\mathbb{H}^{2n+1}(-1)$. 
\end{proof}

\subsection{Proper RYS and non-existence of $Q \cdot R = 0$}
The next theorem is concerned about the non-existence of the curvature property $Q \cdot R = 0$ on $(k,\mu)$-almost Kenmotsu manifolds admiiting proper RYS.
\begin{theorem}
If $M$ be a $(2n + 1)$-dimensional $(k,\mu)$-almost Kenmotsu manifold admitting a proper RYS $(g,\xi,\lambda,\alpha,\beta)$, then the curvaute property $Q \cdot R = 0$ does not hold on $M$.
\end{theorem}
\begin{proof}
Let the curvature property $Q \cdot R = 0$ holds on $M$. Then for any vector fields $X$, $Y$ and $Z$ on $M$, we have
\bea
Q(R(X,Y)Z) - R(QX,Y)Z - R(X,QY)Z - R(X,Y)QZ =0. \label{4.20}
\eea
Using (\ref{4.11}) in (\ref{4.20}), we get
\bea
\nonumber && 2(2n\alpha + 1)R(X,Y)Z + \eta(R(X,Y)Z)\xi -\eta(X)R(\xi,Y)Z \\ &&- \eta(Y)R(X,\xi)Z - \eta(Z)R(X,Y)\xi = 0. \label{4.21}
\eea
With the help of (\ref{4.2}), we obtain
\bea
\eta(R(X,Y)Z) = g(X,Z)\eta(Y) - g(Y,Z)\eta(X). \label{4.22}
\eea
Using (\ref{4.3}) and (\ref{4.22}) in (\ref{4.21}) yields
\bea
(2n\alpha + 1)R(X,Y)Z = \eta(Z)[\eta(X)Y - \eta(Y)X],
\eea
which implies
\bea
R(X,Y)\xi = \frac{1}{2n\alpha + 1}[\eta(X)Y - \eta(Y)X]. \label{4.23}
\eea
Comparing (\ref{4.2}) and k\ref{4.23}), we get $\frac{1}{2n\alpha + 1} = 1$, which implies $\alpha = 0$, a contradiction to the fact that the RYS is proper. This completes the proof.
\end{proof}

\vspace{0.5in}

\noindent \textbf{Author's Information}\\

\noindent {Dibakar Dey\\ Department of Pure Mathematics\\ University of Calcutta\\ 35, Ballygunge Circular Road\\ Kolkata-700019\\West Bengal, India\\ Email: deydibakar3@gmail.com}\\

\end{document}